\documentclass[secthm,seceqn,amsthm,ussrhead,12pt]{amsart}
\usepackage[utf8]{inputenc}
\usepackage[english]{babel}
\usepackage{amssymb,amsmath,amsthm,amsfonts,xcolor,enumerate,hyperref,comment,longtable,cleveref}

\usepackage{times}
\usepackage{cite}
\usepackage{pdflscape}
\usepackage{ulem}
\usepackage[mathcal]{euscript}
\usepackage{tikz}
\usepackage{hyperref}
\usepackage{cancel}
\usepackage{stmaryrd}

\usetikzlibrary{arrows}

\newtheorem*{theoremA}{Theorem A}
\newtheorem*{theoremB}{Theorem B}
\newtheorem*{theoremC}{Theorem C}

\newtheorem{theorem}{Theorem}
\newtheorem{lemma}[theorem]{Lemma}

\newtheorem{definition}[theorem]{Definition}
\newtheorem{corollary}[theorem]{Corollary}

\newenvironment{Proof}[1][Proof.]{\begin{trivlist}
\item[\hskip \labelsep {\bfseries #1}]}{\flushright
$\Box$\end{trivlist}}

\usepackage{stmaryrd}
\usepackage{xcolor}

\begin{document}
\noindent{\Large
The algebraic and geometric classification of nilpotent\\noncommutative Jordan algebras}
\footnote{
The work was supported by
 RFBR 18-31-20004; FAPESP  18/15712-0, 	19/00192-3. }

   \

   {\bf  Doston Jumaniyozov$^{a}$,
   Ivan   Kaygorodov$^{b}$ \&
   Abror Khudoyberdiyev$^{a, c}$}

\

{\tiny

$^{a}$ National University of Uzbekistan, Tashkent, Uzbekistan.

$^{b}$ CMCC, Universidade Federal do ABC, Santo Andr\'e, Brazil.

$^{c}$ Institute of Mathematics Academy of Science of Uzbekistan, Tashkent,  Uzbekistan.

\smallskip

   E-mail addresses:

\smallskip

Doston Jumaniyozov (jumaniyozovdoston50@gmail.com)

Ivan Kaygorodov (kaygorodov.ivan@gmail.com)

Abror Khudoyberdiyev (khabror@mail.ru)

}

\

\

\noindent{\bf Abstract}:
{\it We give algebraic and geometric classifications of complex $4$-dimensional nilpotent noncommutative Jordan algebras. Specifically, we find that, up to isomorphism, there are only $18$ non-isomorphic nontrivial nilpotent noncommutative Jordan algebras. The corresponding geometric variety is  determined by the Zariski closure of $3$ rigid algebras and $2$ one-parametric families of algebras.}

\

\noindent {\bf Keywords}:
{\it noncommutative Jordan algebras, nilpotent algebras, algebraic classification, central extension, geometric classification, degeneration.}

\

\noindent {\bf MSC2010}: 17A15, 17C55.

\section*{Introduction}
Algebraic classification (up to isomorphism) of algebras of small dimension from a certain variety defined by a family of polynomial identities is a classic problem in the theory of non-associative algebras. There are many results related to algebraic classification of small dimensional algebras in varieties of Jordan, Lie, Leibniz, Zinbiel and other algebras.
Another interesting approach of studying algebras of a fixed dimension is to study them from a geometric point of view (that is, to study degenerations and deformations of these algebras). The results in which the complete information about degenerations of a certain variety is obtained are generally referred to as the geometric classification of the algebras of these variety. There are many results related to algebraic and geometric classification of 
Jordan, Lie, Leibniz, Zinbiel and other algebras \cite{ack,  BC99,  cfk18,  ck13,  casas,  degr3,  usefi1,  degr2,  degr1,  demir, gkk,  GRH, ha16,  hac18,  ikv17,  ikv18,  kkk18,  KLP19,kpv19,  kppv,  kpv,  kv16,  kv17,  krh14,  S90,fkkv, kkp19}.

Noncommutative Jordan algebras were introduced by Albert in \cite{Alb}. He noted that the structure theories of alternative and Jordan algebras share so many nice properties that it is natural to conjecture that these algebras are members of a more general class with a similar theory. So he introduced the variety of noncommutative Jordan algebras defined by the Jordan identity and the flexibility identity. 
Namely, the variety of noncommutative Jordan algebras is defined by the following identities:
\[
\begin{array}{rcl}
x(yx) &=&  (xy)x,\\
x^2(yx) &=&  (x^2y)x\end{array} \]
The class of noncommutative Jordan algebras turned out to be vast: for example, apart from alternative and Jordan algebras it contains quasiassociative algebras, quadratic flexible algebras and anticommutative algebras. However, the structure theory of this class is far from being nice.
Nevertheless, a certain progress was made in the study of structure theory of noncommutative Jordan algebras and superalgebras 
 (see, for example \cite{Sch,McC,McC2,KLP17,popov,spa}). 

In this paper, our goal is to obtain a complete algebraic and geometric description of the variety of all $4$-dimensional nilpotent noncommutative Jordan algebras over the complex field. To do so, we first determine all such $4$-dimensional algebra structures, up to isomorphism (what we call the algebraic classification), and then proceed to determine the geometric properties of the corresponding variety, namely its dimension and description of the irreducible components (the geometric classification). 
Note that, every $2$-step nilpotent algebra is a noncommutative Jordan algebra
and the algebraic and geometric classification of complex $4$-dimensional $2$-step nilpotent algebras can be found in \cite{kppv}.

Our main results are summarized below. 

\begin{theoremA}
There are only $18$ non-isomorphic complex $4$-dimensional nontrivial nilpotent noncommutative Jordan algebras, described explicitly in Appendix~A.
\end{theoremA}

From the geometric point of view, in many cases the irreducible components of the variety are determined by the rigid algebras, i.e., algebras whose orbit closure is an irreducible component. It is worth mentioning that this is not always the case and already in \cite{f68}, Flanigan had shown that the variety of $3$-dimensional nilpotent associative algebras has an irreducible component which does not contain any rigid algebras --- it is instead defined by the closure of a union of a one-parameter family of algebras. Here, we encounter a different situation. Informally, although Theorem~B shows that there are three  {\it generic} algebras and 
two  {\it generic} parametric families in the variety of $4$-dimensional nilpotent noncommutative Jordan algebras. 

\begin{theoremB}
The variety of complex $4$-dimensional  nilpotent noncommutative Jordan algebras has dimension $14$. It is definded by  $3$ rigid algebras and two one-parametric families of algebras, and can be described as the closure of the union of $\mathrm{GL}_4(\mathbb{C})$-orbits of the following algebras ($\alpha\in\mathbb{C}$):
\begin{longtable}{llllllll}

$ \mathfrak{N}_2(\alpha)$  &$:$& $e_1e_1 = e_3$ & $e_1e_2 = e_4$  &$e_2e_1 = -\alpha e_3$ & $e_2e_2 = -e_4$ \\
$\mathfrak{N}_3(\alpha)$  &$:$& $e_1e_1 = e_4$ & $e_1e_2 = \alpha e_4$  & $e_2e_1 = -\alpha e_4$ & $e_2e_2 = e_4$\\         
${\mathcal J}^4_{07}$&$:$& $e_1 e_2 = e_3$ & $e_1 e_3=e_4$  & $e_2 e_1=e_3+e_4$ & $e_2 e_3=e_4$ & $e_3 e_1=e_4$ & $e_3 e_2=e_4$\\
         ${\mathcal J}^4_{17}$&$:$& $e_1 e_1 = e_4$ & $e_1 e_2 = e_3$ & $e_2 e_1=-e_3$ & $e_1 e_3=e_4$ & $e_2 e_2=e_4$ & $e_3 e_1=-e_4$\\
        ${\mathcal J}^4_{18}$&$:$& $e_1 e_1 = e_2$ &$e_1 e_2 = e_3$ & $e_1 e_3=e_4$ & $e_2 e_1=e_3$ & $e_2 e_2=e_4$ & $e_3 e_1=e_4.$
\end{longtable}

\end{theoremB}

\section{The algebraic classification of nilpotent noncommutative Jordan algebras}
\subsection{Method of classification of nilpotent algebras}

Throughout this paper, we use the notations and methods well written in \cite{hac16,cfk18},
which we have adapted for the noncommutative Jordan case with some modifications.
Further in this section we give some important definitions.

Let $({\bf A}, \cdot)$ be a noncommutative Jordan  algebra over  $\mathbb C$
and $\mathbb V$ be a vector space over ${\mathbb C}$. The $\mathbb C$-linear space ${\rm Z^{2}}\left(
\bf A,\mathbb V \right) $ is defined as the set of all  bilinear maps $\theta  \colon {\bf A} \times {\bf A} \longrightarrow {\mathbb V}$ such that
\[\theta(x, yz) + \theta(z, yx) = \theta(xy, z) + \theta(zy, x)\]
\[\theta(xt, yz) + \theta(tx, yz) + \theta(tz, yx) + \theta(xz, yt) + \theta(zt, yx) + \theta(zx, yt) = \]
\[
\theta((xt)y, z) + \theta((tx)y, z) + \theta((xz)y, t) + \theta((tz)y, x) +
\theta((zx)y, t) + \theta((zt)y, x)\]

These elements will be called {\it cocycles}. For a
linear map $f$ from $\bf A$ to  $\mathbb V$, if we define $\delta f\colon {\bf A} \times
{\bf A} \longrightarrow {\mathbb V}$ by $\delta f  (x,y ) =f(xy )$, then $\delta f\in {\rm Z^{2}}\left( {\bf A},{\mathbb V} \right) $. We define ${\rm B^{2}}\left({\bf A},{\mathbb V}\right) =\left\{ \theta =\delta f\ : f\in {\rm Hom}\left( {\bf A},{\mathbb V}\right) \right\} $.
We define the {\it second cohomology space} ${\rm H^{2}}\left( {\bf A},{\mathbb V}\right) $ as the quotient space ${\rm Z^{2}}
\left( {\bf A},{\mathbb V}\right) \big/{\rm B^{2}}\left( {\bf A},{\mathbb V}\right) $.

\

Let $\operatorname{Aut}({\bf A}) $ be the automorphism group of  ${\bf A} $ and let $\phi \in \operatorname{Aut}({\bf A})$. For $\theta \in
{\rm Z^{2}}\left( {\bf A},{\mathbb V}\right) $ define  the action of the group $\operatorname{Aut}({\bf A}) $ on ${\rm H^{2}}\left( {\bf A},{\mathbb V}\right) $ by $\phi \theta (x,y)
=\theta \left( \phi \left( x\right) ,\phi \left( y\right) \right) $.  It is easy to verify that
 ${\rm B^{2}}\left( {\bf A},{\mathbb V}\right) $ is invariant under the action of $\operatorname{Aut}({\bf A}).$
 So, we have an induced action of  $\operatorname{Aut}({\bf A})$  on ${\rm H^{2}}\left( {\bf A},{\mathbb V}\right)$.

\

Let $\bf A$ be a noncommutative Jordan  algebra of dimension $m$ over  $\mathbb C$ and ${\mathbb V}$ be a $\mathbb C$-vector
space of dimension $k$. For $\theta \in {\rm Z^{2}}\left(
{\bf A},{\mathbb V}\right) $, define on the linear space ${\bf A}_{\theta } = {\bf A}\oplus {\mathbb V}$ the
bilinear product `` $\left[ -,-\right] _{{\bf A}_{\theta }}$'' by $\left[ x+x^{\prime },y+y^{\prime }\right] _{{\bf A}_{\theta }}=
 xy +\theta(x,y) $ for all $x,y\in {\bf A},x^{\prime },y^{\prime }\in {\mathbb V}$.
The algebra ${\bf A}_{\theta }$ is called an $k$-{\it dimensional central extension} of ${\bf A}$ by ${\mathbb V}$. One can easily check that ${\bf A_{\theta}}$ is a noncommutative Jordan
algebra if and only if $\theta \in {\rm Z^2}({\bf A}, {\mathbb V})$.

Call the
set $\operatorname{Ann}(\theta)=\left\{ x\in {\bf A}:\theta \left( x, {\bf A} \right)+ \theta \left({\bf A} ,x\right) =0\right\} $
the {\it annihilator} of $\theta $. We recall that the {\it annihilator} of an  algebra ${\bf A}$ is defined as
the ideal $\operatorname{Ann}(  {\bf A} ) =\left\{ x\in {\bf A}:  x{\bf A}+ {\bf A}x =0\right\}$. Observe
 that
$\operatorname{Ann}\left( {\bf A}_{\theta }\right) =(\operatorname{Ann}(\theta) \cap\operatorname{Ann}({\bf A}))
 \oplus {\mathbb V}$.

\

The following result shows that every algebra with a non-zero annihilator is a central extension of a smaller-dimensional algebra.

\begin{lemma}
Let ${\bf A}$ be an $n$-dimensional noncommutative Jordan algebra such that $\dim (\operatorname{Ann}({\bf A}))=m\neq0$. Then there exists, up to isomorphism, a unique $(n-m)$-dimensional noncommutative Jordan  algebra ${\bf A}'$ and a bilinear map $\theta \in {\rm Z^2}({\bf A}, {\mathbb V})$ with $\operatorname{Ann}({\bf A})\cap\operatorname{Ann}(\theta)=0$, where $\mathbb V$ is a vector space of dimension m, such that ${\bf A} \cong {{\bf A}'}_{\theta}$ and
 ${\bf A}/\operatorname{Ann}({\bf A})\cong {\bf A}'$.
\end{lemma}

\begin{proof}
Let ${\bf A}'$ be a linear complement of $\operatorname{Ann}({\bf A})$ in ${\bf A}$. Define a linear map $P \colon {\bf A} \longrightarrow {\bf A}'$ by $P(x+v)=x$ for $x\in {\bf A}'$ and $v\in\operatorname{Ann}({\bf A})$, and define a multiplication on ${\bf A}'$ by $[x, y]_{{\bf A}'}=P(x y)$ for $x, y \in {\bf A}'$.
For $x, y \in {\bf A}$, we have
\[P(xy)=P((x-P(x)+P(x))(y- P(y)+P(y)))=P(P(x) P(y))=[P(x), P(y)]_{{\bf A}'}. \]

Since $P$ is a homomorphism $P({\bf A})={\bf A}'$ is a noncommutative Jordan algebra and
 ${\bf A}/\operatorname{Ann}({\bf A})\cong {\bf A}'$, which gives us the uniqueness. Now, define the map $\theta \colon {\bf A}' \times {\bf A}' \longrightarrow\operatorname{Ann}({\bf A})$ by $\theta(x,y)=xy- [x,y]_{{\bf A}'}$.
  Thus, ${\bf A}'_{\theta}$ is ${\bf A}$ and therefore $\theta \in {\rm Z^2}({\bf A}, {\mathbb V})$ and $\operatorname{Ann}({\bf A})\cap\operatorname{Ann}(\theta)=0$.
\end{proof}

\begin{definition}
Let ${\bf A}$ be an algebra and $I$ be a subspace of $\operatorname{Ann}({\bf A})$. If ${\bf A}={\bf A}_0 \oplus I$
then $I$ is called an {\it annihilator component} of ${\bf A}$.
\end{definition}
\begin{definition}
A central extension of an algebra $\bf A$ without annihilator component is called a {\it non-split central extension}.
\end{definition}

Our task is to find all central extensions of an algebra $\bf A$ by
a space ${\mathbb V}$.  In order to solve the isomorphism problem we need to study the
action of $\operatorname{Aut}({\bf A})$ on ${\rm H^{2}}\left( {\bf A},{\mathbb V}
\right) $. To do that, let us fix a basis $e_{1},\ldots ,e_{s}$ of ${\mathbb V}$, and $
\theta \in {\rm Z^{2}}\left( {\bf A},{\mathbb V}\right) $. Then $\theta $ can be uniquely
written as $\theta \left( x,y\right) =
\displaystyle \sum_{i=1}^{s} \theta _{i}\left( x,y\right) e_{i}$, where $\theta _{i}\in
{\rm Z^{2}}\left( {\bf A},\mathbb C\right) $. Moreover, $\operatorname{Ann}(\theta)=\operatorname{Ann}(\theta _{1})\cap\operatorname{Ann}(\theta _{2})\cap\cdots \cap\operatorname{Ann}(\theta _{s})$. Furthermore, $\theta \in
{\rm B^{2}}\left( {\bf A},{\mathbb V}\right) $\ if and only if all $\theta _{i}\in {\rm B^{2}}\left( {\bf A},
\mathbb C\right) $.
It is not difficult to prove (see \cite[Lemma 13]{hac16}) that given a noncommutative Jordan algebra ${\bf A}_{\theta}$, if we write as
above $\theta \left( x,y\right) = \displaystyle \sum_{i=1}^{s} \theta_{i}\left( x,y\right) e_{i}\in {\rm Z^{2}}\left( {\bf A},{\mathbb V}\right) $ and
$\operatorname{Ann}(\theta)\cap \operatorname{Ann}\left( {\bf A}\right) =0$, then ${\bf A}_{\theta }$ has an
annihilator component if and only if $\left[ \theta _{1}\right] ,\left[
\theta _{2}\right] ,\ldots ,\left[ \theta _{s}\right] $ are linearly
dependent in ${\rm H^{2}}\left( {\bf A},\mathbb C\right) $.

\;

Let ${\mathbb V}$ be a finite-dimensional vector space over $\mathbb C$. The {\it Grassmannian} $G_{k}\left( {\mathbb V}\right) $ is the set of all $k$-dimensional
linear subspaces of $ {\mathbb V}$. Let $G_{s}\left( {\rm H^{2}}\left( {\bf A},\mathbb C\right) \right) $ be the Grassmannian of subspaces of dimension $s$ in
${\rm H^{2}}\left( {\bf A},\mathbb C\right) $. There is a natural action of $\operatorname{Aut}({\bf A})$ on $G_{s}\left( {\rm H^{2}}\left( {\bf A},\mathbb C\right) \right) $.
Let $\phi \in \operatorname{Aut}({\bf A})$. For $W=\left\langle
\left[ \theta _{1}\right] ,\left[ \theta _{2}\right] ,\dots,\left[ \theta _{s}
\right] \right\rangle \in G_{s}\left( {\rm H^{2}}\left( {\bf A},\mathbb C
\right) \right) $ define $\phi W=\left\langle \left[ \phi \theta _{1}\right]
,\left[ \phi \theta _{2}\right] ,\dots,\left[ \phi \theta _{s}\right]
\right\rangle $. We denote the orbit of $W\in G_{s}\left(
{\rm H^{2}}\left( {\bf A},\mathbb C\right) \right) $ under the action of $\operatorname{Aut}({\bf A})$ by $\operatorname{Orb}(W)$. Given
\[
W_{1}=\left\langle \left[ \theta _{1}\right] ,\left[ \theta _{2}\right] ,\dots,
\left[ \theta _{s}\right] \right\rangle ,W_{2}=\left\langle \left[ \vartheta
_{1}\right] ,\left[ \vartheta _{2}\right] ,\dots,\left[ \vartheta _{s}\right]
\right\rangle \in G_{s}\left( {\rm H^{2}}\left( {\bf A},\mathbb C\right)
\right),
\]
we easily have that if $W_{1}=W_{2}$, then $ \bigcap\limits_{i=1}^{s}\operatorname{Ann}(\theta _{i})\cap \operatorname{Ann}\left( {\bf A}\right) = \bigcap\limits_{i=1}^{s}
\operatorname{Ann}(\vartheta _{i})\cap\operatorname{Ann}( {\bf A}) $, and therefore we can introduce
the set
\[
{\bf T}_{s}({\bf A}) =\left\{ W=\left\langle \left[ \theta _{1}\right] ,
\left[ \theta _{2}\right] ,\dots,\left[ \theta _{s}\right] \right\rangle \in
G_{s}\left( {\rm H^{2}}\left( {\bf A},\mathbb C\right) \right) : \bigcap\limits_{i=1}^{s}\operatorname{Ann}(\theta _{i})\cap\operatorname{Ann}({\bf A}) =0\right\},
\]
which is stable under the action of $\operatorname{Aut}({\bf A})$.

\

Now, let ${\mathbb V}$ be an $s$-dimensional linear space and let us denote by
${\bf E}\left( {\bf A},{\mathbb V}\right) $ the set of all {\it non-split $s$-dimensional central extensions} of ${\bf A}$ by
${\mathbb V}$. By above, we can write
\[
{\bf E}\left( {\bf A},{\mathbb V}\right) =\left\{ {\bf A}_{\theta }:\theta \left( x,y\right) = \sum_{i=1}^{s}\theta _{i}\left( x,y\right) e_{i} \ \ \text{and} \ \ \left\langle \left[ \theta _{1}\right] ,\left[ \theta _{2}\right] ,\dots,
\left[ \theta _{s}\right] \right\rangle \in {\bf T}_{s}({\bf A}) \right\} .
\]
We also have the following result, which can be proved as in \cite[Lemma 17]{hac16}.

\begin{lemma}
 Let ${\bf A}_{\theta },{\bf A}_{\vartheta }\in {\bf E}\left( {\bf A},{\mathbb V}\right) $. Suppose that $\theta \left( x,y\right) =  \displaystyle \sum_{i=1}^{s}
\theta _{i}\left( x,y\right) e_{i}$ and $\vartheta \left( x,y\right) =
\displaystyle \sum_{i=1}^{s} \vartheta _{i}\left( x,y\right) e_{i}$.
Then the noncommutative Jordan algebras ${\bf A}_{\theta }$ and ${\bf A}_{\vartheta } $ are isomorphic
if and only if
$$\operatorname{Orb}\left\langle \left[ \theta _{1}\right] ,
\left[ \theta _{2}\right] ,\dots,\left[ \theta _{s}\right] \right\rangle =
\operatorname{Orb}\left\langle \left[ \vartheta _{1}\right] ,\left[ \vartheta
_{2}\right] ,\dots,\left[ \vartheta _{s}\right] \right\rangle .$$
\end{lemma}

This shows that there exists a one-to-one correspondence between the set of $\operatorname{Aut}({\bf A})$-orbits on ${\bf T}_{s}\left( {\bf A}\right) $ and the set of
isomorphism classes of ${\bf E}\left( {\bf A},{\mathbb V}\right) $. Consequently we have a
procedure that allows us, given a noncommutative Jordan algebra ${\bf A}'$ of
dimension $n-s$, to construct all non-split central extensions of ${\bf A}'$. This procedure is:

\; \;

{\centerline {\textsl{Procedure}}}

\begin{enumerate}
\item For a given noncommutative Jordan algebra ${\bf A}'$ of dimension $n-s $, determine ${\rm H^{2}}( {\bf A}',\mathbb {C}) $, $\operatorname{Ann}({\bf A}')$ and $\operatorname{Aut}({\bf A}')$.

\item Determine the set of $\operatorname{Aut}({\bf A}')$-orbits on $T_{s}({\bf A}') $.

\item For each orbit, construct the noncommutative Jordan algebra associated with a
representative of it.
\end{enumerate}

\

\subsection{Notations}
Let ${\bf A}$ be a noncommutative Jordan algebra with
a basis $e_{1},e_{2},\dots,e_{n}$. Then by $\Delta _{ij}$\ we denote the
bilinear form
$\Delta _{ij} \colon {\bf A}\times {\bf A}\longrightarrow \mathbb C$
with $\Delta _{ij}\left( e_{l},e_{m}\right) = \delta_{il}\delta_{jm}$.
Then the set $\left\{ \Delta_{ij}:1\leq i, j\leq n\right\} $ is a basis for the space of
the bilinear forms on ${\bf A}$. Then every $\theta \in
{\rm Z^{2}}\left( {\bf A},\mathbb C\right) $ can be uniquely written as $
\theta = \displaystyle \sum_{1\leq i,j\leq n} c_{ij}\Delta _{{i}{j}}$, where $
c_{ij}\in \mathbb C$.
Let us fix the following notations:

$$\begin{array}{lll}
{\mathcal J}^{i*}_j& \mbox{---}& j \mbox{th }i\mbox{-dimensional nilpotent  noncommutative Jordan algebra with identity $xyz=0$} \\
{\mathcal J}^i_j& \mbox{---}& j \mbox{th }i\mbox{-dimensional nilpotent "pure" noncommutative Jordan algebra (without identity $xyz=0$)} \\
{\mathfrak{N}}_i& \mbox{---}& i\mbox{-dimensional algebra with zero product} \\
({\bf A})_{i,j}& \mbox{---}& j \mbox{th }i\mbox{-dimensional central extension of }\bf A. \\
\end{array}$$

\subsection{The algebraic classification of  $3$-dimensional nilpotent noncommutative Jordan algebras}
There are no nontrivial $1$-dimensional nilpotent Jordan algebras.
There is only one nontrivial $2$-dimensional nilpotent Jordan algebra
(it is the non-split central extension of $1$-dimensional algebra with zero product):

$$\begin{array}{ll llll}
{\mathcal J}^{2*}_{01} &:& (\mathfrak{N}_1)_{2,1} &:& e_1 e_1 = e_2.\\
\end{array}$$

Thanks to \cite{cfk18} we have the description of all central extensions of  ${\mathcal J}^{2*}_{01}$ and $\mathfrak{N}_2$.
Choosing the Jordan algebras from the central extensions of these algebras,
we have the classification of all non-split $3$-dimensional nilpotent Jordan algebras:

$$\begin{array}{ll llllllllllll}
{\mathcal J}^{3*}_{02} &:& (\mathfrak{N}_2)_{3,1} &:& e_1 e_2 = e_3, &  e_2 e_1=e_3; \\
{\mathcal J}^{3*}_{03} &:& (\mathfrak{N}_2)_{3,2} &:& e_1 e_2=e_3, & e_2 e_1=-e_3;   \\
{\mathcal J}^{3*}_{04}(\lambda) &:& (\mathfrak{N}_2)_{3,3} &:&
e_1 e_1 = \lambda e_3,  & e_2 e_1=e_3,  & e_2 e_2=e_3; \\
{\mathcal J}^3_{01} &:& ({\mathcal J}^{2*}_{01} )_{3,1} &:& e_1 e_1 = e_2, & e_1 e_2=e_3, & e_2 e_1= e_3.
\end{array} $$

\subsection{$1$-dimensional central extensions of $3$-dimensional  nilpotent noncommutative Jordan algebras}
\label{centrext}
\subsubsection{The description of second cohomology spaces of  $3$-dimensional nilpotent noncommutative Jordan algebras:}

\
In the following table we give the description of the second cohomology space of  $3$-dimensional nilpotent noncommutative Jordan algebras

{\tiny
$$
\begin{array}{|l|l|l|l|}
\hline
\bf A  & {\rm Z^{2}}\left( {\bf A}\right)  & {\rm B^2}({\bf A}) & {\rm H^2}({\bf A}) \\
\hline
\hline
{\mathcal J}^{3*}_{01} &  \langle
\Delta_{11},\Delta_{12}+\Delta_{21},\Delta_{13},\Delta_{23}+\Delta_{32}, \Delta_{31}, \Delta_{33} \rangle
&\langle \Delta_{11} \rangle&
\langle [\Delta_{12}]+[\Delta_{21}],[\Delta_{13}],[\Delta_{23}+\Delta_{32}], [\Delta_{31}], [\Delta_{33}] \rangle\\
\hline

{\mathcal J}^{3*}_{02} &  \langle \Delta_{11}, \Delta_{12},  \Delta_{13}+\Delta_{31}, \Delta_{21}, \Delta_{22}, \Delta_{23}+\Delta_{32} \rangle
& \langle \Delta_{12}+\Delta_{21} \rangle &  \langle [\Delta_{11}],  [\Delta_{13}]+[\Delta_{31}], [\Delta_{21}], [\Delta_{22}], [\Delta_{23}]+[\Delta_{32}] \rangle \\
\hline
{\mathcal J}^{3*}_{03} & \langle \Delta_{11},\Delta_{12}, \Delta_{13}-\Delta_{31},  \Delta_{21}, \Delta_{22}, \Delta_{23}-\Delta_{32} \rangle
& \langle \Delta_{12}-\Delta_{21} \rangle &  \langle [\Delta_{11}], [\Delta_{13}]-[\Delta_{31}], [\Delta_{21}], [\Delta_{22}], [\Delta_{23}]-[\Delta_{32}] \rangle \\
\hline
{\mathcal J}^{3*}_{04}(\lambda) & \langle \Delta_{11}, \Delta_{12}, \Delta_{21}, \Delta_{22} \rangle & \langle \lambda\Delta_{11}+\Delta_{21} + \Delta_{22} \rangle &  \langle [\Delta_{11}], [\Delta_{12}], [\Delta_{22}] \rangle \\

\hline
{\mathcal J}^3_{01} &\langle \Delta_{11}, \Delta_{12}+\Delta_{21}, \Delta_{13}+\Delta_{22}+\Delta_{31} \rangle& \langle \Delta_{11}, \Delta_{12}+\Delta_{21} \rangle &  \langle [\Delta_{13}]+[\Delta_{22}]+[\Delta_{31}] \rangle \\
\hline

\end{array}$$
}
where ${\mathcal J}^{3*}_{01}={\mathcal J}^{2*}_{01}\oplus{\mathbb C e_3}.$

\subsubsection{Central extensions of ${\mathcal J}^{3*}_{01}$}

Let us use the following notations:
\[
\nabla_1=[\Delta_{12}]+[\Delta_{21}], \quad \nabla_2=[\Delta_{13}], \quad  \nabla_3=[\Delta_{23}]+[\Delta_{32}], \quad \nabla_4=[\Delta_{31}], \quad \nabla_{5}=[\Delta_{33}]. \]

The automorphism group of ${\mathcal J}^{3*}_{01}$ consists of invertible matrices of the form

\[\phi=\begin{pmatrix}
x & 0 & 0\\
y & x^2 & u\\
z & 0 & v
\end{pmatrix}. \]

Since

\[ \phi^T\begin{pmatrix}
0 & \alpha_1 & \alpha_2\\
\alpha_1 & 0 & \alpha_{3}\\
\alpha_4 & \alpha_3 & \alpha_5
\end{pmatrix}\phi =\begin{pmatrix}
\alpha^* & \alpha_1^* & \alpha_2^*\\
\alpha_1^* & 0 & \alpha_{3}^*\\
\alpha_4^* & \alpha_3^* & \alpha_5^*
\end{pmatrix}\]
the action of $\operatorname{Aut} (\mathcal{J}_{01}^{3*})$ on subspace
$\Big\langle \sum\limits_{i=1}^5 \alpha_i\nabla_i \Big\rangle$ is given by
$\Big\langle \sum\limits_{i=1}^5 \alpha_i^*\nabla_i \Big\rangle,$
where

\[\begin{array}{rcl}
\alpha^*_1&=&x^2(x \alpha_1+z\alpha_{3});\\
\alpha^*_2&=&u(x\alpha_1+z\alpha_{3}) + v (x \alpha_2 + y \alpha_{3} + z \alpha_5);\\
\alpha^*_3&=&x^2 v\alpha_{3};\\
\alpha^*_4&=&u(x\alpha_1+z\alpha_{3}) + v (x \alpha_4 + y \alpha_{3} + z \alpha_5);\\
\alpha^*_5&=&2uv\alpha_{3}+v^2 \alpha_{5}.
\end{array}\]

It is easy to see that the elements $\nabla_1$ and  $\alpha_2\nabla_2 +\alpha_4\nabla_4+\alpha_5\nabla_5$ give algebras with $2$-dimensional annihilator, which were described before.
Since we are interested only in new algebras, we have the following cases:

\begin{enumerate}
\item $\alpha_2 \neq \alpha_4,$ then:
\begin{enumerate}
    \item  if $\alpha_3\neq0,$
    then choosing
        $x=\frac{\alpha_4-\alpha_2}{\alpha_3}, \ y=\frac{x(\alpha_5 \alpha_1 - \alpha_2\alpha_3)} {\alpha_{3}^{2}}, \  z=-\frac{x\alpha_{1}}{\alpha_{3}}, \   u=-\frac{v\alpha_{5}}{2\alpha_{3}},$
        we have the representative
        $\langle \nabla_3+ \nabla_4 \rangle.$

\item if $\alpha_3=0, $ then $\alpha_1 \neq 0$ and:
  \begin{enumerate}
  \item if $\alpha_5 \neq 0,$ then choosing
            $x=\frac{(\alpha_4-\alpha_2)^2}{\alpha_1\alpha_5},  \ v = \frac{x(\alpha_4-\alpha_2)}{\alpha_5}, \ u =-\frac{v(x\alpha_2 + z\alpha_5)}{x\alpha_1},$
        we have the representative
    $\langle \nabla_1+ \nabla_4+\nabla_5 \rangle.$
    \item if  $\alpha_5 = 0,$ then choosing $v = \frac{x^2\alpha_1}{\alpha_4-\alpha_2}, \ u =-\frac{v\alpha_2}{\alpha_1},$
        we have the representative
    $\langle \nabla_1+ \nabla_4 \rangle.$
    \end{enumerate}
\end{enumerate}

\item $\alpha_2 = \alpha_4,$ then:
\begin{enumerate}
    \item if $\alpha_3\neq0,$
    then choosing
        $ z=-\frac{x\alpha_{1}}{\alpha_{3}}, \ y=- \frac{x\alpha_2 + z\alpha_5} {\alpha_{3}}, \   u=-\frac{v\alpha_{5}}{2\alpha_{3}},$
        we have the representative
        $\langle \nabla_3\rangle.$

\item if  $\alpha_3=0, $ then $\alpha_1 \neq 0$ and choosing
            $u =-\frac{v(x\alpha_2 + z\alpha_5)}{x\alpha_1},$
        we get the representatives $\langle \nabla_1+\nabla_5 \rangle$ and $\langle \nabla_1 \rangle$
        depending on whether $\alpha_5=0$ or not. Since $\nabla_1$ give an algebra with $2$-dimensional annihilator, we have only representatives $\langle \nabla_1+\nabla_5 \rangle.$

\end{enumerate}

\end{enumerate}

Now we have five new 4-dimensional nilpotent noncommutative Jordan algebras constructed from ${\mathcal J}^{3*}_{01}$:
\[{\mathcal J}^{4}_{02}, {\mathcal J}^{4}_{03}, {\mathcal J}^{4}_{04}, {\mathcal J}^{4}_{05}, {\mathcal J}^{4}_{06}.\]

The multiplication tables of these algebras can be found in Appendix.

\subsubsection{Central extensions of ${\mathcal J}^{3*}_{02}$}

Let us use the following notations:
\[
\nabla_1=[\Delta_{11}], \quad \nabla_2=[\Delta_{13}]+[\Delta_{31}], \quad \nabla_3=[\Delta_{21}],\quad \nabla_4=[\Delta_{22}], \quad  \nabla_5=[\Delta_{23}]+[\Delta_{32}] . \]

The automorphism group of ${\mathcal J}^{3*}_{02}$ consists of invertible matrices of the form
\[\phi_1=\begin{pmatrix}
x & 0 & 0\\
0 & v & 0\\
z & w & xv
\end{pmatrix} \quad \text{or} \quad  \phi_2=\begin{pmatrix}
0 & u & 0\\
y & 0 & 0\\
z & w & uy
\end{pmatrix}\]

Since
\[ \phi_1^T\begin{pmatrix}
\alpha_1 & 0 & \alpha_2\\
\alpha_3 & \alpha_4 & \alpha_{5}\\
\alpha_2 & \alpha_5 & 0
\end{pmatrix}\phi_1 =
\begin{pmatrix}
\alpha_{1}x^2+2\alpha_{2}xz & \alpha^* & \alpha_{2}x^2v \\
\alpha^* + \alpha_3xv& \alpha_{4}v^2+2\alpha_{5}vw & \alpha_{5}xv^2 \\
\alpha_{2}x^2v &\alpha_{5}xv^2& 0
\end{pmatrix},\]
the action of $\operatorname{Aut} (\mathcal{J}_{01}^{3*})$ on subspace
$\Big\langle \sum\limits_{i=1}^5 \alpha_i\nabla_i \Big\rangle$ is given by
$\Big\langle \sum\limits_{i=1}^5 \alpha_i^*\nabla_i \Big\rangle,$
where

\[\begin{array}{rcl}
\alpha^*_1&=&\alpha_{1}x^2+2\alpha_{2}xz;\\
\alpha^*_2&=&\alpha_{2}x^2v;\\
\alpha^*_3&=&\alpha_3xv;\\
\alpha^*_4&=&\alpha_{4}v^2+2\alpha_{5}vw;\\
\alpha^*_5&=&\alpha_{5}xv^2.
\end{array}\]

Note that $(\alpha_2; \alpha_5)\neq (0; 0)$ and since
\[ \phi_2^T\begin{pmatrix}
\alpha_1 & 0 & \alpha_2\\
\alpha_3 & \alpha_4 & \alpha_{5}\\
\alpha_2 & \alpha_5 & 0
\end{pmatrix}\phi_2 =
\begin{pmatrix}
\alpha_{4}y^2+2\alpha_{5}yz & \alpha^* & \alpha_{5}y^2u \\
\alpha^* - \alpha_3yu& \alpha_{1}u^2+2\alpha_{2}uw & \alpha_{2}yu^2 \\
\alpha_{5}y^2u &\alpha_{2}yu^2& 0
\end{pmatrix},\]
we can always assume $\alpha_2 \neq 0.$ 

Thus, we have the following cases:

\begin{enumerate}
\item $\alpha_3 \neq 0,$ then:
\begin{enumerate}
\item if $\alpha_5 \neq 0,$ then choosing $x=\frac{\alpha_3}{\alpha_2},$ $v=\frac{\alpha_3}{\alpha_5},$ $z=-\frac{x\alpha_1}{2\alpha_2},$ $w=-\frac{v\alpha_4}{2\alpha_5},$ we have the representative $ \langle \nabla_2+\nabla_3 +\nabla_5 \rangle. $

\item if $\alpha_5 = 0,$ then: 
\begin{enumerate}
 \item if $\alpha_4\neq 0,$  then choosing $x=\frac {\alpha_3} {\alpha_2},$ $v=\frac {\alpha_3^2} {\alpha_2\alpha_4},$ $z=-\frac{x\alpha_1}{2\alpha_2},$
        we have the representative $\langle  \nabla_2+\nabla_3 +\nabla_4\rangle.$
 \item if $\alpha_4= 0,$  then choosing $x=\frac {\alpha_3} {\alpha_2},$ $z=-\frac{x\alpha_1}{2\alpha_2},$
        we have the representative $\langle  \nabla_2+\nabla_3\rangle.$
\end{enumerate}

\end{enumerate}
\item $\alpha_3 = 0,$ then:
\begin{enumerate}
\item if $\alpha_5 \neq 0,$ then choosing  $v=\frac{x\alpha_2}{\alpha_5},$ $z=-\frac{x\alpha_1}{2\alpha_2},$ $w=-\frac{v\alpha_4}{2\alpha_5},$ we have the representative $ \langle \nabla_2+ \nabla_5 \rangle. $

\item if $\alpha_5 = 0,$  then choosing $z=-\frac{x\alpha_1}{2\alpha_2},$
        we have the representative $\langle  \nabla_2 \rangle$ and $\langle  \nabla_2 +\nabla_4\rangle$ depending on whether
        $\alpha_4=0$ or not.

\end{enumerate}

\end{enumerate}

Now we have six new 4-dimensional nilpotent noncommutative Jordan algebras constructed from ${\mathcal J}^{3*}_{02}$:

\[{\mathcal J}^{4}_{07}, {\mathcal J}^{4}_{08}, {\mathcal J}^{4}_{09}, {\mathcal J}^{4}_{10}, {\mathcal J}^{4}_{11}, {\mathcal J}^{4}_{12}.\]

The multiplication tables of these algebras can be found in Appendix.

\subsubsection{Central extensions of ${\mathcal J}^{3*}_{03}$}

Let us use the following notations:
\[
\nabla_1=[\Delta_{11}], \quad \nabla_2=[\Delta_{13}]-[\Delta_{31}], \quad \nabla_3=[\Delta_{21}],\quad \nabla_4=[\Delta_{22}], \quad  \nabla_5=[\Delta_{23}]-[\Delta_{32}] . \]

The automorphism group of ${\mathcal J}^{3*}_{03}$ consists of invertible matrices of the form

\[\phi=\begin{pmatrix}
x & u & 0\\
y & v & 0\\
z & w & xv-yu
\end{pmatrix}. \]

Since

\[ \phi^T\begin{pmatrix}
\alpha_1 & 0 & \alpha_2\\
\alpha_3 & \alpha_4 & \alpha_{5}\\
-\alpha_2 & -\alpha_5 & 0
\end{pmatrix}\phi =\begin{pmatrix}
\alpha_1^* & \alpha^* & \alpha_2^*\\
-\alpha^* + \alpha_3^* & \alpha_4^* & \alpha_{5}^*\\
-\alpha_2^* & -\alpha_5^* & 0
\end{pmatrix}\]
the action of $\operatorname{Aut} (\mathcal{J}_{01}^{3*})$ on subspace
$\Big\langle \sum\limits_{i=1}^5 \alpha_i\nabla_i \Big\rangle$ is given by
$\Big\langle \sum\limits_{i=1}^5 \alpha_i^*\nabla_i \Big\rangle,$
where

\[\begin{array}{rcl}
\alpha^*_1&=&\alpha_{1}x^2+\alpha_{3}xy+\alpha_{4}y^2;\\
\alpha^*_2&=&(\alpha_{2}x+\alpha_{5}y)(xv-yu);\\
\alpha^*_3&=&2\alpha_{1}ux+\alpha_{3}(xv+yu)+2\alpha_{4}yv;\\
\alpha^*_4&=&\alpha_{1}u^2+\alpha_{3}uv+\alpha_{4}v^2;\\
\alpha^*_5&=&(\alpha_{2}u+\alpha_{5}v)(xv-yu).
\end{array}\]

Since $(\alpha_2; \alpha_5)\neq (0; 0)$, we have the following cases:
\begin{enumerate}
\item  $\alpha_5=0,$ then choosing $u=0,$ we have  
\[\begin{array}{rcl}
\alpha^*_1&=&\alpha_{1}x^2+\alpha_{3}xy+\alpha_{4}y^2;\\
\alpha^*_2&=&\alpha_{2}x^2v;\\
\alpha^*_3&=&\alpha_{3}xv+2\alpha_{4}yv;\\
\alpha^*_4&=&\alpha_{4}v^2;\\
\alpha^*_5&=&0.
\end{array}\]

\begin{enumerate}
\item  $\alpha_4=\alpha_3= \alpha_1 =0,$ then we have the representative $ \langle \nabla_2 \rangle. $

\item $\alpha_4=\alpha_3=0,$ $\alpha_1\neq 0,$ then choosing $v=\frac {\alpha_1} {\alpha_2},$  
        we have the representative $\langle  \nabla_1+\nabla_2 \rangle. $
        
\item $\alpha_4=0,$ $\alpha_3\neq 0,$ then choosing $x=\frac {\alpha_3} {\alpha_2},$ $y=-\frac {x\alpha_1} {\alpha_3},$ 
        we have the representative $\langle  \nabla_2+\nabla_3 \rangle. $

\item $\alpha_4\neq0,$ then:
\begin{enumerate}
    \item if $4\alpha_1\alpha_4-\alpha_3^2=0,$ then choosing
    $y=-\frac{x\alpha_3}{2\alpha_4}, v=\frac{x^2\alpha_2}{\alpha_4},$
    we have the representative $\langle \nabla_2+\nabla_4 \rangle.$

\item if $4\alpha_1\alpha_4-\alpha_3^2\neq 0,$ then choosing
    $v=\frac{4\alpha_1\alpha_4-\alpha_3^2}{4\alpha_2\alpha_4}, x=\sqrt{\frac{v\alpha_4}{\alpha_2}}, y=-\frac{x\alpha_3}{2\alpha_4},$
    we have the representative $\langle \nabla_1+ \nabla_2+\nabla_4 \rangle.$
\end{enumerate}
\end{enumerate}

\item  $\alpha_5\neq0,$ then choosing $v=-\frac{u\alpha_2}{\alpha_5},$ we have
\[\begin{array}{rcl}
\alpha^*_1&=&\alpha_{1}x^2+\alpha_{3}xy+\alpha_{4}y^2;\\
\alpha^*_2&=&-\frac{(\alpha_{2}x+\alpha_{5}y)^2}{\alpha_5}u;\\
\alpha^*_3&=&(\frac{2\alpha_1\alpha_5-\alpha_2\alpha_3}{\alpha_5}x+\frac{\alpha_3\alpha_5-2\alpha_2\alpha_4}{\alpha_5}y)u;\\
\alpha^*_4&=&\frac{\alpha_1\alpha_{5}^2-\alpha_2\alpha_3\alpha_5+\alpha_{2}^2\alpha_4}{\alpha_{5}^2}u^2;\\
\alpha^*_5&=&0.
\end{array}\]

\begin{enumerate}
\item  $\alpha_4=0,$

\begin{enumerate}
\item  $\alpha_3=0,$ $\alpha_1=0,$
then we have the representative $ \langle \nabla_2 \rangle. $

\item  $\alpha_3=0,$ $\alpha_1\neq0,$ then choosing $x=0,$ $u=-\frac{y^2\alpha_5}{\alpha_1},$ we have the representative  $\langle \nabla_2+ \nabla_4 \rangle.$

\item  $\alpha_3\neq 0,$ $\alpha_1\alpha_5-\alpha_2\alpha_3=0,$ then choosing $x=0,$ $y=-\frac {\alpha_3}{\alpha_5},$ have the representative $\langle \nabla_2+ \nabla_3 \rangle.$

\item  $\alpha_3\neq 0,$ $\alpha_1\alpha_5-\alpha_2\alpha_3\neq0,$ then choosing $x=\frac {\alpha_3^2}{2(\alpha_1\alpha_5-\alpha_2\alpha_3)} i,$ $y=\frac {x(\alpha_2\alpha_3-2\alpha_1\alpha_5)}{\alpha_3\alpha_5},$  $u=\frac {\alpha_3^2}{2(\alpha_1\alpha_5-\alpha_2\alpha_3)},$
 have the representative $\langle \nabla_1 + \nabla_2+ \nabla_4 \rangle.$
\end{enumerate}

\item  $\alpha_4\neq0,$

\begin{enumerate}
\item  $\alpha_3\alpha_5-2\alpha_2\alpha_4=0,$ $\alpha_1\alpha_5^2 - \alpha_2^2\alpha_4=0,$ then choosing $u=-\frac{\alpha_4}{\alpha_5}$
then we have the representative $ \langle \nabla_1 + \nabla_2 \rangle. $

\item  $\alpha_3\alpha_5-2\alpha_2\alpha_4=0,$ $\alpha_1\alpha_5^2 - \alpha_2^2\alpha_4\neq0,$ choosing $x=0,$ $y=\frac{\sqrt{\alpha_4(\alpha_1\alpha_{5}^2-\alpha_{2}^2\alpha_4)}}{\alpha_5^2},$ $u=-\frac{\alpha_4}{\alpha_5}$ we have the representative  $\langle \nabla_1+\nabla_2+ \nabla_4 \rangle.$

\item  $\alpha_3\alpha_5-2\alpha_2\alpha_4\neq0,$ $\alpha_1\alpha_{5}^2-\alpha_2\alpha_3\alpha_5+\alpha_{2}^2\alpha_4= 0,$ then choosing $x=\frac {\alpha_4}{\alpha_5^2},$ $y=\frac{x(\alpha_2\alpha_4-\alpha_3\alpha_5)}{\alpha_4\alpha_5},$ have the representative $\langle \nabla_2+ \nabla_3 \rangle.$

\item  $\alpha_3\alpha_5-2\alpha_2\alpha_4\neq0,$ $\alpha_1\alpha_{5}^2-\alpha_2\alpha_3\alpha_5+\alpha_{2}^2\alpha_4\neq0,$ 
 then choosing
 $y = \frac{\alpha_2\alpha_3-2\alpha_1\alpha_5}{\alpha_3\alpha_5-2\alpha_2\alpha_4}x,$ we get the representatives  $\langle \alpha^*_1\nabla_1+ \alpha^*_2\nabla_2+\alpha^*_4\nabla_4 \rangle,$ where
 \[\begin{array}{rcl}
\alpha^*_1&=&\frac {(\alpha_1\alpha_{5}^2-\alpha_2\alpha_3\alpha_5+\alpha_{2}^2\alpha_4)(4\alpha_1\alpha_4-\alpha_3^2)} {(\alpha_3\alpha_5-2\alpha_2\alpha_4)^2}x^2;\\
\alpha^*_2&=& -\frac {4(\alpha_1\alpha_{5}^2-\alpha_2\alpha_3\alpha_5+\alpha_{2}^2\alpha_4)^2} {\alpha_5(\alpha_3\alpha_5-2\alpha_2\alpha_4)^2} x^2u;\\
\alpha^*_4&=&\frac{\alpha_1\alpha_{5}^2-\alpha_2\alpha_3\alpha_5+\alpha_{2}^2\alpha_4}{\alpha_{5}^2}u^2.
\end{array}\] 
 
 It gives us the representatives  $\langle \nabla_2+ \nabla_4 \rangle$ and
$\langle \nabla_1+ \nabla_2+ \nabla_4 \rangle$ depending on whether $4\alpha_1\alpha_4-\alpha_3^2=0$ or not.

\end{enumerate}

\end{enumerate}
\end{enumerate}

Thus, we have the following representatives of distinct orbits
$\langle\nabla_2\rangle,  \langle\nabla_1+ \nabla_2\rangle,  \langle\nabla_2+ \nabla_3\rangle, \langle\nabla_2+ \nabla_4\rangle$ and $\langle\nabla_1+ \nabla_2+ \nabla_4 \rangle,$
which give five new 4-dimensional nilpotent noncommutative Jordan algebras constructed from ${\mathcal J}^{3*}_{03}$:
\[{\mathcal J}^{4}_{13}, {\mathcal J}^{4}_{14}, {\mathcal J}^{4}_{15}, {\mathcal J}^{4}_{16}, {\mathcal J}^{4}_{17}.\]

The multiplication tables of these algebras can be found in Appendix.

\subsection{The algebraic classification of $4$-dimensional nilpotent noncommutative Jordan algebras}

Now we are ready to state the main result of this part of the paper.
The proof of the present theorem is  based on the classification of $3$-dimensional nilpotent noncommutative Jordan algebras and the results of Section \ref{centrext}.

\begin{theoremA}
There are only $18$ non-isomorphic complex $4$-dimensional nontrivial nilpotent noncommutative Jordan algebras, described explicitly in Appendix~A.
\end{theoremA}

\section{The geometric classification of nilpotent noncommutative Jordan algebras}

\subsection{Definitions and notation}
Given an $n$-dimensional vector space $\mathbb V$, the set ${\rm Hom}(\mathbb V \otimes \mathbb V,\mathbb V) \cong \mathbb V^* \otimes \mathbb V^* \otimes \mathbb V$
is a vector space of dimension $n^3$. This space has the structure of the affine variety $\mathbb{C}^{n^3}.$ Indeed, let us fix a basis $e_1,\dots,e_n$ of $\mathbb V$. Then any $\mu\in {\rm Hom}(\mathbb V \otimes \mathbb V,\mathbb V)$ is determined by $n^3$ structure constants $c_{ij}^k\in\mathbb{C}$ such that
$\mu(e_i\otimes e_j)=\sum\limits_{k=1}^nc_{ij}^ke_k$. A subset of ${\rm Hom}(\mathbb V \otimes \mathbb V,\mathbb V)$ is {\it Zariski-closed} if it can be defined by a set of polynomial equations in the variables $c_{ij}^k$ ($1\le i,j,k\le n$).

Let $T$ be a set of polynomial identities.
The set of algebra structures on $\mathbb V$ satisfying polynomial identities from $T$ forms a Zariski-closed subset of the variety ${\rm Hom}(\mathbb V \otimes \mathbb V,\mathbb V)$. We denote this subset by $\mathbb{L}(T)$.
The general linear group $\mathrm{GL}(\mathbb V)$ acts on $\mathbb{L}(T)$ by conjugations:
$$ (g * \mu )(x\otimes y) = g\mu(g^{-1}x\otimes g^{-1}y)$$
for $x,y\in \mathbb V$, $\mu\in \mathbb{L}(T)\subset {\rm Hom}(\mathbb V \otimes\mathbb V, \mathbb V)$ and $g\in \mathrm{GL}(\mathbb V)$.
Thus, $\mathbb{L}(T)$ is decomposed into $\mathrm{GL}(\mathbb V)$-orbits that correspond to the isomorphism classes of algebras.
Let $O(\mu)$ denote the orbit of $\mu\in\mathbb{L}(T)$ under the action of $\mathrm{GL}(\mathbb V)$ and $\overline{O(\mu)}$ denote the Zariski closure of $O(\mu)$.

Let $\mathcal A$ and $\mathcal B$ be two $n$-dimensional algebras satisfying the identities from $T$, and let $\mu,\lambda \in \mathbb{L}(T)$ represent $\mathcal A$ and $\mathcal B$, respectively.
We say that $\mathcal A$ degenerates to $\mathcal B$ and write $\mathcal A\to \mathcal B$ if $\lambda\in\overline{O(\mu)}$.
Note that in this case we have $\overline{O(\lambda)}\subset\overline{O(\mu)}$. Hence, the definition of a degeneration does not depend on the choice of $\mu$ and $\lambda$. If $\mathcal A\not\cong \mathcal B$, then the assertion $\mathcal A\to \mathcal B$ is called a {\it proper degeneration}. We write $\mathcal A\not\to \mathcal B$ if $\lambda\not\in\overline{O(\mu)}$.

Let $\mathcal A$ be represented by $\mu\in\mathbb{L}(T)$. Then  $\mathcal A$ is  {\it rigid} in $\mathbb{L}(T)$ if $O(\mu)$ is an open subset of $\mathbb{L}(T)$.
 Recall that a subset of a variety is called irreducible if it cannot be represented as a union of two non-trivial closed subsets.
 A maximal irreducible closed subset of a variety is called an {\it irreducible component}.
It is well known that any affine variety can be represented as a finite union of its irreducible components in a unique way.
The algebra $\mathcal A$ is rigid in $\mathbb{L}(T)$ if and only if $\overline{O(\mu)}$ is an irreducible component of $\mathbb{L}(T)$.

Given the spaces $U$ and $W$, we write simply $U>W$ instead of $\dim \,U>\dim \,W$.



\subsection{Method of the description of  degenerations of algebras}

In the present work we use the methods applied to Lie algebras in \cite{BC99,GRH,GRH2,S90}.
First of all, if $\mathcal A\to \mathcal B$ and $\mathcal A\not\cong \mathcal B$, then $\mathfrak{Der}(\mathcal A)<\mathfrak{Der}(\mathcal B)$, where $\mathfrak{Der}(\mathcal A)$ is the Lie algebra of derivations of $\mathcal A$. We compute the dimensions of algebras of derivations and check the assertion $\mathcal A\to \mathcal B$ only for such $\mathcal A$ and $\mathcal B$ that $\mathfrak{Der}(\mathcal A)<\mathfrak{Der}(\mathcal B)$.


To prove degenerations, we construct families of matrices parametrized by $t$. Namely, let $\mathcal A$ and $\mathcal B$ be two algebras represented by the structures $\mu$ and $\lambda$ from $\mathbb{L}(T)$ respectively. Let $e_1,\dots, e_n$ be a basis of $\mathbb  V$ and $c_{ij}^k$ ($1\le i,j,k\le n$) be the structure constants of $\lambda$ in this basis. If there exist $a_i^j(t)\in\mathbb{C}$ ($1\le i,j\le n$, $t\in\mathbb{C}^*$) such that $E_i^t=\sum\limits_{j=1}^na_i^j(t)e_j$ ($1\le i\le n$) form a basis of $\mathbb V$ for any $t\in\mathbb{C}^*$, and the structure constants of $\mu$ in the basis $E_1^t,\dots, E_n^t$ are such rational functions $c_{ij}^k(t)\in\mathbb{C}[t]$ that $c_{ij}^k(0)=c_{ij}^k$, then $\mathcal A\to \mathcal B$.
In this case  $E_1^t,\dots, E_n^t$ is called a {\it parametrized basis} for $\mathcal A\to \mathcal B$.

Since the variety of $4$-dimensional nilpotent noncommutative Jordan algebras  contains infinitely many non-isomorphic algebras, we have to do some additional work.
Let $\mathcal A(*):=\{\mathcal A(\alpha)\}_{\alpha\in I}$ be a series of algebras, and let $\mathcal B$ be another algebra. Suppose that for $\alpha\in I$, $\mathcal A(\alpha)$ is represented by the structure $\mu(\alpha)\in\mathbb{L}(T)$ and $B\in\mathbb{L}(T)$ is represented by the structure $\lambda$. Then we say that $\mathcal A(*)\to \mathcal B$ if $\lambda\in\overline{\{O(\mu(\alpha))\}_{\alpha\in I}}$, and $\mathcal A(*)\not\to \mathcal B$ if $\lambda\not\in\overline{\{O(\mu(\alpha))\}_{\alpha\in I}}$.

Let $\mathcal A(*)$, $\mathcal B$, $\mu(\alpha)$ ($\alpha\in I$) and $\lambda$ be as above. To prove $\mathcal A(*)\to \mathcal B$ it is enough to construct a family of pairs $(f(t), g(t))$ parametrized by $t\in\mathbb{C}^*$, where $f(t)\in I$ and $g(t)\in \mathrm{GL}(\mathbb V)$. Namely, let $e_1,\dots, e_n$ be a basis of $\mathbb V$ and $c_{ij}^k$ ($1\le i,j,k\le n$) be the structure constants of $\lambda$ in this basis. If we construct $a_i^j:\mathbb{C}^*\to \mathbb{C}$ ($1\le i,j\le n$) and $f: \mathbb{C}^* \to I$ such that $E_i^t=\sum\limits_{j=1}^na_i^j(t)e_j$ ($1\le i\le n$) form a basis of $\mathbb V$ for any  $t\in\mathbb{C}^*$, and the structure constants of $\mu_{f(t)}$ in the basis $E_1^t,\dots, E_n^t$ are such rational functions $c_{ij}^k(t)\in\mathbb{C}[t]$ that $c_{ij}^k(0)=c_{ij}^k$, then $\mathcal A(*)\to \mathcal B$. In this case  $E_1^t,\dots, E_n^t$ and $f(t)$ are called a parametrized basis and a {\it parametrized index} for $\mathcal A(*)\to \mathcal B$, respectively.

We now explain how to prove $\mathcal A(*)\not\to\mathcal  B$.
Note that if $\mathfrak{Der} \ \mathcal A(\alpha)  > \mathfrak{Der} \  \mathcal B$ for all $\alpha\in I$ then $\mathcal A(*)\not\to\mathcal B$.
One can use also the following generalization of Lemma from \cite{GRH}, whose proof is the same as the proof of Lemma.

\begin{lemma}\label{gmain}
Let $\mathfrak{B}$ be a Borel subgroup of $\mathrm{GL}(\mathbb V)$ and $\mathcal{R}\subset \mathbb{L}(T)$ be a $\mathfrak{B}$-stable closed subset.
If $\mathcal A(*) \to \mathcal B$ and for any $\alpha\in I$ the algebra $\mathcal A(\alpha)$ can be represented by a structure $\mu(\alpha)\in\mathcal{R}$, then there is $\lambda\in \mathcal{R}$ representing $\mathcal B$.
\end{lemma}

\subsection{The geometric classification of small dimensional noncommutative Jordan algebras}
Thanks to the description of all degenerations and orbit closures in the varieties of all $2$-dimensional algebras \cite{kv16} and all $3$-dimensional nilpotent algebras \cite{fkkv}, 
we have the following statements.

\begin{lemma}
The variety of complex $2$-dimensional noncommutative Jordan algebras has two irreducible components defined by the following algebras:

\begin{longtable}{llllllll}

         ${\bf E}_{1}(0,0,0,0)$ &$:$& $e_1 e_1 = e_1$ & $e_2 e_2=e_2$\\

        ${\bf E}_{5}(\alpha)$ &$:$& $e_1 e_1 = e_1$ & $e_1 e_2=(1-\alpha)e_1+\alpha e_2$  & $e_2 e_1=\alpha e_1+(1-\alpha) e_2$ & $e_2 e_2 = e_2$ \\
        \end{longtable}

\end{lemma}

\begin{lemma}
The variety of complex $3$-dimensional nilpotent noncommutative Jordan algebras has two irreducible components defined by the following algebras:

\begin{longtable}{llllllll}

       ${\mathcal J}^{3*}_{04}(\lambda)$&$:$& $e_1 e_1 = \lambda e_3$  & $e_2 e_1=e_3$  & $e_2 e_2=e_3$ \\
       ${\mathcal J}^3_{01}$  &$:$& $e_1 e_1 = e_2$ & $e_1 e_2=e_3$ & $e_2 e_1= e_3$
        \end{longtable}
    \end{lemma}    
        
\subsection{Rigid $n$-dimensional nilpotent noncommutative Jordan algebra}
As follows from \cite{Alb}, every one generated noncommutative Jordan algebra is associative;
and from \cite{kv17}, we conclude that an one generated algebra can not stay in the orbit closure of a non-one generated algebra (or a family of non-one generated algebras).
Then, summarizing we have the following lemma.

\begin{lemma}\label{le}
Any $n$-dimensional one generated nilpotent Jordan algebra is commutative and it is isomorphic to the following algebra
\[{\mathcal J}^{n} : e_{i}e_{j}=e_{i+j}, \quad 1 \leq i, j \leq{n}, \quad  i+j \leq{n}.\]
The algebra ${\mathcal J}^n$ is rigid in the variety of complex $n$-dimensional nilpotent noncommutative Jordan algebras.
\end{lemma}

\subsection{The geometric classification of $4$-dimensional nilpotent noncommutative Jordan algebras}
The main result of the present section is the following theorem.

\begin{theoremB}
The variety of complex $4$-dimensional  nilpotent noncommutative Jordan algebras has dimension $14$. It is definded by  $3$ rigid algebras and two one-parametric families of algebras, and can be described as the closure of the union of $\mathrm{GL}_4(\mathbb{C})$-orbits of the following algebras ($\alpha\in\mathbb{C}$):
\begin{longtable}{llllllll}

$ \mathfrak{N}_2(\alpha)$  &$:$& $e_1e_1 = e_3$ & $e_1e_2 = e_4$  &$e_2e_1 = -\alpha e_3$ & $e_2e_2 = -e_4$ \\
$\mathfrak{N}_3(\alpha)$  &$:$& $e_1e_1 = e_4$ & $e_1e_2 = \alpha e_4$  & $e_2e_1 = -\alpha e_4$ & $e_2e_2 = e_4$\\         
${\mathcal J}^4_{07}$&$:$& $e_1 e_2 = e_3$ & $e_1 e_3=e_4$  & $e_2 e_1=e_3+e_4$ & $e_2 e_3=e_4$ & $e_3 e_1=e_4$ & $e_3 e_2=e_4$\\
         ${\mathcal J}^4_{17}$&$:$& $e_1 e_1 = e_4$ & $e_1 e_2 = e_3$ & $e_2 e_1=-e_3$ & $e_1 e_3=e_4$ & $e_2 e_2=e_4$ & $e_3 e_1=-e_4$\\
        ${\mathcal J}^4_{18}$&$:$& $e_1 e_1 = e_2$ &$e_1 e_2 = e_3$ & $e_1 e_3=e_4$ & $e_2 e_1=e_3$ & $e_2 e_2=e_4$ & $e_3 e_1=e_4.$
\end{longtable}

\end{theoremB}

\begin{Proof}
The variety of $4$-dimensional noncommutative Jordan algebras has two principal subvarieties:
$4$-dimensional $2$-step nilpotent algebras and $4$-dimensional Jordan algebras.

Recall that the full description of the degeneration system of $4$-dimensional $2$-step nilpotent algebras was given in \cite{kppv}.
Using the cited result, we can see that the variety of $4$-dimensional non-pure noncommutative Jordan algebras has two irreducible components given by the following
families of algebras:

$$\begin{array}{lllllll}
\mathfrak{N}_2(\alpha)  & e_1e_1 = e_3, &e_1e_2 = e_4,  &e_2e_1 = -\alpha e_3, &e_2e_2 = -e_4 \\
\mathfrak{N}_3(\alpha)  & e_1e_1 = e_4, &e_1e_2 = \alpha e_4,  &e_2e_1 = -\alpha e_4, &e_2e_2 = e_4,  &e_3e_3 = e_4.
\end{array}$$

Thanks to \cite{fkkv}, 
the variety of $4$-dimensional nilpotent Jordan algebras is defined by two rigid algebras: 
${\mathcal J}^4_{10}$  and ${\mathcal J}^4_{18}.$

Now we can prove that the variety of $4$-dimensional nilpotent noncommutative Jordan algebras has five irreducible components.
Thanks to Lemma \ref{le}, the algebra ${\mathcal J}^4_{18}$ is rigid.
One can easily compute that
  $ \mathfrak{Der} \ {\mathcal J}^4_{07}=2$ and $\mathfrak{Der} \ {\mathcal J}^4_{17}=4.$
  It is follows that $ {\mathcal J}^4_{17} \not \to {\mathcal J}^4_{07}.$

The list of all necessary degenerations is given in Table B (see Appendix A)
and all needs arguments for non-degenerations are given below:

\begin{longtable}{|rcl|l|}
 
\hline
\multicolumn{3}{|c|}{\textrm{Non-degeneration}} & \multicolumn{1}{|c|}{\textrm{Arguments}}\\
\hline


${\mathcal J}^4_{07}$ & $\not \to$ & $\mathfrak{N}_2(\alpha), \mathfrak{N}_3(\alpha), {\mathcal J}^4_{17}$ & 
${\mathcal R}=
\left\{
A_2^2 \subseteq A_4, A_3^2 = 0,  c_{12}^3=c_{21}^3, c_{13}^4=c_{31}^4, c_{23}^4=c_{32}^4  
\right\}$ \\ 
\hline

${\mathcal J}^4_{17}$ & $\not \to$ & $\mathfrak{N}_2(\alpha), \mathfrak{N}_3(\alpha)$ & 
${\mathcal R}=
\left\{
A_3^2 = 0,  c_{11}^2=c_{11}^3=c_{22}^3=0  
\right\}$ \\ 
\hline

\end{longtable}

\end{Proof}

\section*{Appendix A.}

\begin{longtable}{llllllll}

\multicolumn{8}{c}{  {\bf Table A.}
{\it The list of $4$-dimensional "pure" nilpotent noncommutative Jordan algebras.}} \\ \\
        ${\mathcal J}^4_{01}$&$:$& $e_1 e_1 = e_2$ & $e_1 e_2=e_3$  & $e_2 e_1=e_3$ \\
        ${\mathcal J}^4_{02}$&$:$& $e_1 e_1 = e_2$ & $e_2 e_3=e_4$  & $e_3 e_2=e_4$ & $e_3 e_1=e_4$ \\
        ${\mathcal J}^4_{03}$&$:$& $e_1 e_1 = e_2$ & $e_1 e_2=e_4$  & $e_2 e_1=e_4$ & $e_3 e_1=e_4$ & $e_3 e_3=e_4$\\
        ${\mathcal J}^4_{04}$&$:$& $e_1 e_1 = e_2$ & $e_1 e_2=e_4$  & $e_2 e_1=e_4$ & $e_3 e_1=e_4$\\
        ${\mathcal J}^4_{05}$&$:$& $e_1 e_1 = e_2$ & $e_2 e_3=e_4$  & $e_3 e_2=e_4$\\
        ${\mathcal J}^4_{06}$&$:$& $e_1 e_1 = e_2$ & $e_1 e_2=e_4$  & $e_2 e_1=e_4$ & $e_3e_3 =e_4$\\
        ${\mathcal J}^4_{07}$&$:$& $e_1 e_2 = e_3$ & $e_1 e_3=e_4$  & $e_2 e_1=e_3+e_4$ & $e_2 e_3=e_4$ & $e_3 e_1=e_4$ & $e_3 e_2=e_4$\\
        ${\mathcal J}^4_{08}$&$:$& $e_1 e_2 = e_3$ & $e_1 e_3= e_4$ & $e_2 e_1=e_3+e_4$ & $e_2 e_2=e_4$ & $e_3 e_1=e_4$\\
        ${\mathcal J}^4_{09}$&$:$& $e_1 e_2 = e_3$ & $e_1 e_3= e_4$ & $e_2 e_1=e_3+e_4$ & $e_3 e_1=e_4$\\
        ${\mathcal J}^4_{10}$&$:$& $e_1 e_2 = e_3$ & $e_1 e_3= e_4$ & $e_2 e_1=e_3$ & $e_2 e_3=e_4$ & $e_3 e_1=e_4$  & $e_3 e_2=e_4$\\
        ${\mathcal J}^4_{11}$&$:$& $e_1 e_2 = e_3$ & $e_1 e_3= e_4$ & $e_2 e_1=e_3$ & $e_3 e_1=e_4$\\
        ${\mathcal J}^4_{12}$&$:$& $e_1 e_2 = e_3$ & $e_1 e_3= e_4$ & $e_2 e_1=e_3$ & $e_2 e_2=e_4$ & $e_3 e_1=e_4$\\
        ${\mathcal J}^4_{13}$&$:$& $e_1 e_2 = e_3$ & $e_1 e_3=e_4$  & $e_2 e_1=-e_3$  &$e_3 e_1=-e_4$ \\
        ${\mathcal J}^4_{14}$&$:$& $e_1 e_1 = e_4$ & $e_1 e_2 = e_3$ & $e_1 e_3=e_4$ & $e_2 e_1=-e_3$ & $e_3 e_1=-e_4$ \\
        ${\mathcal J}^4_{15}$&$:$& $e_1 e_2 = e_3$ &  $e_1 e_3=e_4$ & $e_2 e_1=-e_3+e_4$ & $e_3 e_1=-e_4$ \\
        ${\mathcal J}^4_{16}$&$:$& $e_1 e_2 = e_3$ &  $e_1 e_3=e_4$ & $e_2 e_1=-e_3$ & $e_2 e_2=e_4$ & $e_3 e_1=-e_4$ \\
        ${\mathcal J}^4_{17}$&$:$& $e_1 e_1 = e_4$ & $e_1 e_2 = e_3$ & $e_1 e_3=e_4$ & $e_2 e_1=-e_3$ & $e_2 e_2=e_4$ & $e_3 e_1=-e_4$\\
        ${\mathcal J}^4_{18}$&$:$& $e_1 e_1 = e_2$ &$e_1 e_2 = e_3$ & $e_1 e_3=e_4$ & $e_2 e_1=e_3$ & $e_2 e_2=e_4$ & $e_3 e_1=e_4$
\end{longtable}

\begin{longtable}{|lll|llll|}

\multicolumn{7}{c}{ \mbox{ {\bf Table B.}
{\it Degenerations of nilpotent noncommutative Jordan algebras of dimension $4$.}}} \\
\multicolumn{7}{c}{}\\

\hline
${\mathcal J}^{4}_{07} $&$\to$&$ {\mathcal J}^{4}_{02}$  & $E_1^t=-\frac 1 2 e_1+\frac 1 2 e_2 - \frac 1 2 e_3$ & $E_2^t=-\frac 1 2e_3 - \frac 1 2 e_4$&  $E_3^t= t e_2$& $ E_4^t=-\frac 1 2 t e_4$ \\
      
      \hline
      ${\mathcal J}^{4}_{08} $&$\to$&$ {\mathcal J}^{4}_{03}$  & 
    $E_1^t=te_1+\frac{1}{2}e_2$  &    
      $E_2^t=t e_3 +\frac{2t+1}{4}e_4$  & 
      $E_3^t=t e_2+\frac{2t-1}{4}e_3$ &
      $E_4^t=t^{2}e_4$\\
      
  \hline
      ${\mathcal J}^{4}_{03} $&$\to$&$ {\mathcal J}^{4}_{04}$  & $E_1^t=t e_1$ &   $E_2^t=t^{2}e_2$ &  $E_3^t=t^{2}e_3$ &  $E_4^t=t^{3}e_4$ \\

\hline
      ${\mathcal J}^{4}_{07} $&$\to$&$ {\mathcal J}^{4}_{08}$  & 
      \multicolumn{2}{l}{$E_1^t=(t^2+1)e_1+t^2e_2+\frac{t}{2}e_3 $} &  
      \multicolumn{2}{l|}{$E_2^t=t e_2 +\frac{t^2+1}{2}e_3$}\\  
      &&& \multicolumn{2}{l}{$E_3^t=t(t^2+1)e_3+\frac{1}{2}e_4$}& \multicolumn{2}{l|}{$E_4^t= t(t^2+1) e_4$} \\
    \hline
        ${\mathcal J}^{4}_{07} $&$\to$&$ {\mathcal J}^{4}_{09}$  & $E_1^t=e_1$ &$   E_2^t=t e_2$ &   $E_3^t=t e_3$ &  $E_4^t=t e_4$ \\
    \hline
        ${\mathcal J}^{4}_{07} $&$\to$&$ {\mathcal J}^{4}_{10}$  & $E_1^t=t^{-1} e_1$ &  $E_2^t=t^{-1} e_2$ &  $E_3^t=t^{-2}e_3$ &  $E_4^t=t^{-3}e_4$ \\    

    \hline
    ${\mathcal J}^{4}_{14} $&$\to$&$ {\mathcal J}^{4}_{13}$  & $E_1^t=e_1 $&$   E_2^t=t^{-1}e_2$ &$   E_3^t=t^{-1}e_3 $& $   E_4^t=t^{-1}e_4$ \\
    
    \hline
   
    ${\mathcal J}^{4}_{17} $&$\to$&$ {\mathcal J}^{4}_{14}$  & $E_1^t=t^{-1}e_1$ & $E_2^t=e_2$ &$  E_3^t=t^{-1}e_3 $& $E_4^t=t^{-2}e_4$ \\
    
              \hline
    ${\mathcal J}^{4}_{17} $&$\to$&$ {\mathcal J}^{4}_{15}$  & $E_1^t=2i e_1-2e_2$&$  E_2^t=t e_2$&$   E_3^t=2i t e_3-2t e_4$ &$ E_4^t=-4t e_4$ \\
    
              \hline
    ${\mathcal J}^{4}_{17} $&$\to$&$ {\mathcal J}^{4}_{16}$  & $E_1^t=t^{-1}e_1$ & $  E_2^t=t^{-2}e_2$ & $  E_3^t=t^{-3}e_3$ & $   E_4^t=t^{-4}e_4$ \\
    \hline

\end{longtable}

\end{document}